\documentclass[12pt]{article}
\usepackage{amsmath,amssymb,amscd,amsthm}
\usepackage{graphics,amsmath,amssymb,amsthm,mathrsfs}
\usepackage{amssymb,amsfonts}
\usepackage[all,arc]{xy}
\usepackage{enumerate}
\usepackage{mathrsfs}
\usepackage{amsmath,cite,amsthm}

\newtheorem{theorem}{Theorem}[section]
\newtheorem{prop}[theorem]{Proposition}
\newtheorem{lemma}[theorem]{Lemma}

\newtheorem{remark}[theorem]{Remark}

\newcommand{\average}{-\!\!\!\!\!\!\int}

\begin{document}

\title
{\bf Convergence Rates \\ in Homogenization of Stokes Systems}

\author{Shu Gu\footnote{Supported in part by NSF grant DMS-1161154}}

\date{ }

\maketitle

\begin{abstract}
This paper studies the convergence rates in $L^2$ and $H^1$ of Dirichelt problems for Stokes systems with rapidly oscillating periodic coefficients, without any regularity assumptions on the coefficients.


\end{abstract}


\medskip


\section{Introduction and Main Results}
\setcounter{equation}{0}

\quad\  The purpose of this paper is to study the convergence rates of Dirichlet problems for Stokes systems with rapidly oscillating periodic coefficients. More precisely,  we consider the following Dirichlet problem for Stokes systems associated with matrix $A$,
\begin{equation}\label{DirichletStokes}
\left\{
\begin{aligned}
\mathcal{L}_\varepsilon(u_\varepsilon)+\nabla p_\varepsilon &= F &\qquad &\text{ in }\Omega,\\
\text{ div }u_\varepsilon &=g &\qquad &\text{ in }\Omega,\\
u_\varepsilon& =f &\qquad &\text{ on }\partial\Omega,
\end{aligned}
\right.	
\end{equation}
with the compatibility condition
\begin{equation}\label{Compatibility}
\int_\Omega g -\int_{\partial \Omega}f\cdot n=0,
\end{equation}
where $n$ denotes the outward unit normal to $\partial \Omega$ and
 $\Omega\subset \mathbb{R}^d$ is a bounded domain.
 We note that the Dirichlet problem (\ref{DirichletStokes}) is used in the modeling of flows in porous media.
 Here $\varepsilon >0$ is a small parameter and the operator $\mathcal{L}_\varepsilon$ is defined by
\begin{equation}\label{Operator}
\mathcal{L}_\varepsilon= -\text{div}(A(x/\varepsilon)\nabla)=-\frac{\partial}{\partial x_i}\bigg[a_{ij}^{\alpha\beta}\big(\frac{x}{\varepsilon}\big)\frac{\partial}{\partial x_j}\bigg]
\end{equation} with $1\le i,j,\alpha,\beta \le d$ (the summation convention is used throughout). We will assume that the coefficient matrix $A(y)=(a_{ij}^{\alpha\beta}(y))$ is real, bounded measurable, and satisfies the ellipticity condition:
\begin{equation}\label{Ellipticity}
\mu |\xi|^2 \le a_{ij}^{\alpha\beta}(y)\xi_i^\alpha\xi_j^\beta \le \frac{1}{\mu} |\xi|^2 \qquad \text{for }y\in \mathbb{R}^d \text{ and }\xi=(\xi_i^\alpha) \in \mathbb{R}^{d\times d}, 	
\end{equation} 
where $\mu>0$. We also assume that $A(y)$ satisfies the periodicity condition,
\begin{equation}\label{Periodicity}
A(y+z)=A(y) \qquad \text{ for }y\in \mathbb{R}^d\text{ and }z\in \mathbb{Z}^d.
\end{equation}
No symmetry condition on $A(y)$ is needed. A function satisfying (\ref{Periodicity}) will be called 1-periodic.   

By the homogenization theory of Stokes systems (see \cite{Lions,GuShen15}), 
under suitable conditions on $F$, $f$ and $g$,
it is known that 
$$
u_\varepsilon \rightharpoonup u_0 \quad \text{  weakly in  }
H^1(\Omega;\mathbb{R}^d) \quad \text{ and  }\quad p_\varepsilon 
-\average_\Omega p_\varepsilon\rightharpoonup p_0-\average_\Omega p_0 \quad \text{ weakly in  }L^2(\Omega),
$$
where $(u_0, p_0)\in H^1(\Omega;\mathbb{R}^d)\times L^2(\Omega)$ is the weak solution of the  homogenized problem with constant coefficients,
\begin{equation}\label{DirichletStokes0}
\left\{
\begin{aligned}
\mathcal{L}_0(u_0)+\nabla p_0 &= F &\qquad &\text{ in }\Omega,\\
\text{ div }u_0 &=g &\qquad &\text{ in }\Omega,\\
u_0 & =f &\qquad &\text{ on }\partial\Omega.
\end{aligned}
\right.	
\end{equation}
The primary purpose of this paper is to investigate the rate of convergence of $\|u_\varepsilon-u_0\|_{L^2(\Omega)}$, as $\varepsilon \rightarrow 0$. The following is the main result of the paper. 
 
\begin{theorem}\label{theorem1.2}
Let $\Omega$ be a bounded $C^{1,1}$ domain. Suppose that $A$ satisfies the ellipticity condition (\ref{Ellipticity})
and periodicity condition (\ref{Periodicity}). Given $g\in H^1(\Omega)$ and $f\in H^{3/2}(\partial\Omega;\mathbb{R}^d)$ satisfying the compatibility condition(\ref{Compatibility}), for $F\in L^2(\Omega;\mathbb{R}^d)$, 
let $(u_\varepsilon,p_\varepsilon)$, $(u_0,p_0)$ be weak solutions of Dirichlet problems (\ref{DirichletStokes}), (\ref{DirichletStokes0}), respectively.  Then 
\begin{equation}\label{ethm1.2}
\|u_\varepsilon-u_0\|_{L^2(\Omega)} \le C\varepsilon\|u_0\|_{H^2(\Omega)},
\end{equation}
where the constant $C$ depends only on $d$, $\mu$, and $\Omega$.
\end{theorem}

Theorem \ref{theorem1.2} gives the optimal $O(\varepsilon)$ convergence 
rate for the inverses of the Stokes operators in $L^2$ operator norm.
Indeed, let $T_\varepsilon: F\in L^2_\sigma (\Omega)
\to u_\varepsilon$, where $L^2_\sigma(\Omega)=
\big\{ F\in L^2(\Omega; \mathbb{R}^d):\, \text{div} (F)=0   \text{ in } \Omega \big\}$,
and $u_\varepsilon$ denotes the solution of (\ref{DirichletStokes}) with $F\in L^2_\sigma(\Omega; \mathbb{R}^d)$
and $g=0$, $f=0$. Then it follows from (\ref{ethm1.2}) and the estimate 
$\| u_0\|_{H^2(\Omega)} \le C \| F\|_{L^2(\Omega)}$ that
$$
\| T_\varepsilon -T_0 \|_{L^2_\sigma (\Omega) \to L^2_\sigma (\Omega)} \le C\varepsilon,
$$
where $T_0: F\in L^2_\sigma (\Omega)
\to u_0$.

In this paper we also obtain  $O(\sqrt{\varepsilon})$ rates for a two-scale expansion of $(u_\varepsilon, p_\varepsilon)$
in $H^1\times L^2$.
Let $(\chi,\pi)$ denote the correctors associated with $A$, defined by (\ref{Corrector}), and $S_\varepsilon$ the Steklov smoothing operater defined by (\ref{Steklov}).
 
\begin{theorem}\label{theorem1.1}
Let $\Omega$ be a bounded $C^{1,1}$ domain. Suppose that $A$ satisfies (\ref{Ellipticity})
and (\ref{Periodicity}).
Let $(u_\varepsilon,p_\varepsilon)$ and $(u_0,p_0)$ be the same as in Theorem \ref{theorem1.2}.  Then
 \begin{equation}\label{ethm1.1a}
\|u_\varepsilon-u_0-\varepsilon \chi^\varepsilon S_\varepsilon (\nabla\widetilde{u}_0)\|_{H^1(\Omega)} \le C\sqrt{\varepsilon}\|u_0\|_{H^2(\Omega)},
\end{equation}
where $\chi^\varepsilon (x)=\chi(x/\varepsilon)$ and
$\widetilde{u}_0$ is the extension of $u_0$ defined as in (\ref{defextension}). Moreover, if 
$\int_\Omega p_\varepsilon=\int_\Omega p_0=0$, then
\begin{equation}\label{ethm1.1b}            
\|p_\varepsilon-p_0-\Big\{ \pi^\varepsilon S_\varepsilon (\nabla\widetilde{u}_0)
-\average_\Omega \pi^\varepsilon S_\varepsilon (\nabla\widetilde{u}_0)\Big\}\|_{L^2(\Omega)}\le C\sqrt{\varepsilon}\|u_0\|_{H^2(\Omega)},
\end{equation}
where $\pi^\varepsilon (x)=\pi(x/\varepsilon)$.
The constants $C$ in (\ref{ethm1.1a}) and (\ref{ethm1.1b}) depend only on $d$, $\mu$, and $\Omega$.
\end{theorem}

We now describe the known $L^2$ convergence results on Dirichlet problems for general elliptic equations and systems with rapidly oscillating periodic coefficients. 
Consider the Dirichlet problem for the scalar elliptic equation $\mathcal{L}_\varepsilon
(u_\varepsilon)=-\text{div} \big(A(x/\varepsilon)\nabla u_\varepsilon\big)=F$ in a Lipschitz domain
$\Omega$ with  $u_\varepsilon=f$ on $\partial\Omega$.
It is well known that
\begin{equation}\label{e0.1}
\|u_\varepsilon-u_0\|_{L^2(\Omega)}\le C\varepsilon\left\{\|\nabla^2 u_0\|_{L^2(\Omega)}+\|\nabla u_0\|_{L^\infty(\partial\Omega)}\right\}.
\end{equation}
To see (\ref{e0.1}), one considers the difference between $u_\varepsilon$ and its first order approximation $u_0+\varepsilon\chi^\varepsilon\nabla u_0$ and let
\begin{equation}\label{e0.2}
v_\varepsilon=u_\varepsilon-u_0-\varepsilon\chi^\varepsilon \nabla u_0.
\end{equation}
To correct the boundary data, one further introduces a function $w_\varepsilon$, 
where $w_\varepsilon$ is the solution to the Dirichlet problem: $\mathcal{L}_\varepsilon(w_\varepsilon)=0$ in $\Omega$ and $w_\varepsilon=-\varepsilon\chi^\varepsilon\nabla u_0$ on $\partial\Omega$. 
Using energy estimates, one may show that $\|v_\varepsilon-w_\varepsilon\|_{H^1_0(\Omega)}\le C\varepsilon\|\nabla^2 u_0\|_{L^2(\Omega)}$.
The estimate (\ref{e0.1}) follows  from this and  the estimate
$\|w_\varepsilon\|_{L^\infty(\Omega)}\le C\varepsilon\|\nabla u_0\|_{L^\infty(\partial \Omega)}$, which is obtained by the maximum principle (see e.g. \cite{JikovKozlovOleinik94}). 
More recently,  Griso \cite{Griso04,Griso06} was able to establish the much sharper estimate (\ref{ethm1.2}), using the method of periodic unfolding.
We mention that in the case of scalar elliptic equations with bounded measurable coefficients,
one may also prove (\ref{ethm1.2}) by using the so-called Dirichlet corrector. In fact, it was shown in \cite{KenigLinShen1302} that
\begin{equation}\label{D-corrector-estimate}
\| u_\varepsilon -u_0 - \Big\{ \Phi_\varepsilon -x \Big\}\nabla u_0 \|_{H^1_0(\Omega)}
\le C\varepsilon\| u_0\|_{H^2(\Omega)},
\end{equation}
 where $\Phi_\varepsilon (x)$ is the solution of $\mathcal{L}_\varepsilon (\Phi_\varepsilon)=0$ in $\Omega$ with
 $\Phi_\varepsilon =x$ on $\partial\Omega$.
In the case of elliptic systems, the estimates (\ref{D-corrector-estimate}) and thus (\ref{ethm1.2})
continue to hold under the additional assumption that $A$ is H\"older continuous.
Moreover, if $A$ is H\"older continuous and symmetric, it was proved in \cite{KenigLinShen12} that
\begin{equation}\label{H-1/2}
\| v_\varepsilon \|_{H^{1/2} (\Omega)} \le C \varepsilon \| u_0\|_{H^2(\Omega)}.
\end{equation}
 
 The approaches used in \cite{KenigLinShen12, KenigLinShen1302} rely on the uniform regularity 
 estimates established in \cite{AL8701,KenigShen1101} and do not apply to operators with bounded measurable coefficients.
 Recently, by using the Steklov smoothing operator, T.A. Suslina \cite {Suslina1301,Suslina1302} was able to establish the $O(\varepsilon)$
 estimate (\ref{ethm1.2}) in $L^2$ for a boarder class of elliptic operators, which, in particular, contains the
 elliptic systems $\mathcal{L}_\varepsilon$ in divergence form with coefficients satisfying the
 ellipticity condition $a_{ij}^{\alpha\beta} \xi_i^\alpha\xi_j^\beta \ge \mu |\xi|^2$
 for any $\xi = \big(\xi_i^\alpha \big)\in \mathbb{R}^{m\times d}$.
 Since the correctors $\chi$ may not be bounded in the case of nonsmooth coefficients, 
 the idea is to consider the two-scale expansion 
 \begin{equation}\label{two-scale-1}
 v_\varepsilon =u_\varepsilon -u_0 -\varepsilon \chi^\varepsilon S_\varepsilon (\nabla \widetilde{u}_0),
 \end{equation}
 where $S_\varepsilon$ is a smoothing operator at scale $\varepsilon$ and $\widetilde{u}_0$ 
 an extension of $u_0$ to $\mathbb{R}^d$ (also see  \cite{Pastukhova06,PakhninSuslina13,ZhikovPastukhova05} and their references on the use of $S_\varepsilon$ in homogenization).
 This reduces the problem to the control of the $L^2$ norm of $w_\varepsilon$, 
 where $w_\varepsilon$ is the solution to the Dirichlet problem: $\mathcal{L}_\varepsilon(w_\varepsilon)=0$ in $\Omega$ and $w_\varepsilon=-\varepsilon\chi^\varepsilon S_\varepsilon\nabla (\widetilde{u}_0)$ on $\partial\Omega$.
 Next, one considers 
 $$
 h_\varepsilon=w_\varepsilon -\varepsilon \chi^\varepsilon \theta_\varepsilon S_\varepsilon (\nabla \widetilde{u}_0),
 $$
 where $\theta_\varepsilon$ is a cutoff function supported in an $\varepsilon$ neighborhood of $\partial\Omega$.
 Note that $h_\varepsilon=0$ on $\partial\Omega$ and
 $\mathcal{L}_\varepsilon (h_\varepsilon)$ is supported in an $\varepsilon$ neighborhood of $\partial\Omega$.
 This allows one to approximate $h_\varepsilon$ in the  $L^2$ norm by $h_0$, using an $O(\sqrt{\varepsilon})$
 estimate in $H^1$ and a duality argument, 
 where $\mathcal{L}_0 (h_0)=\mathcal{L}_\varepsilon (h_\varepsilon)$ in $\Omega$ and
 $h_0=0$ on $\partial\Omega$. Finally, one estimates the $L^2$ norm of $h_0$ by another duality argument.
 
 In this paper we extend the approach of Suslina to the case of Stokes systems, 
 which do not fit the standard framework 
 of second-order elliptic systems in divergence form.
 As expected in the study of Stokes or Navies-Stokes systems, the main difficulty 
 is caused by the pressure term $p_\varepsilon$.
 By carefully analyzing the systems for the correctors $(\chi, \pi)$ as well as their dual $(\phi_{kin}^{\alpha\beta}, q_{ij}^\beta)$
 (see Lemmas 3.1 and 3.3),
 we are able to establish the $O(\sqrt{\varepsilon})$ error estimates, given in Theorem \ref{theorem1.1},
for the two-scale expansions of $(u_\varepsilon, p_\varepsilon)$
 in $H^1\times  L^2$. This allows us to use the idea of boundary cutoff and duality argument
 in a manner similar to that  in \cite{Suslina1301}.
 
 The paper is organized as follows.
In Section 2 we recall a few basic properties of the Steklov smoothing operator $S_\varepsilon$ as well as
 the homogenization theory for Stokes systems with periodic coefficients.  In Section 3 we study
 $u_0+\varepsilon\chi^\varepsilon S_\varepsilon \nabla\widetilde{u}_0$
as the first order approximation of $u_\varepsilon$. We introduce the dual correctors $(\Phi, q)$
and use energy estimates to establish the estimate (\ref{ethm1.1a}) in $H^1$. 
In Section 4 we study  the convergence of $p_\varepsilon$
and prove the error estimate  (\ref{ethm1.1b}) for the two-scale expansion of the pressure term.
 Finally, our main theorem Theorem \ref{theorem1.2} is proved  in Section 5.
 This is done by using the idea of boundary cutoff and duality, and by
 applying error estimates obtained in Sections 3 and 4 to the adjoint systems.

Throughout this paper, we denote $Y=[0,1)^d$ and the $L^1$ average of $f$ over the set $E$ by
$$
\average_E f= \frac{1}{|E|}\int_E f.
$$
We will use $C$ to denote constants that may depend on $d$, $\mu$, or $\Omega$, but never on $\varepsilon$.
\bf{Acknowledgement.} \rm The author would like to thank referees for their very helpful comments and suggestions.
\section{Preliminaries}
\setcounter{equation}{0}
\subsection{Smoothing in Steklov's sense}
Let $S_\varepsilon$ be the operator on $L^2(\mathbb{R}^d)$ given by 
\begin{equation}\label{Steklov}
(S_\varepsilon u)(x)=\average_Y u(x-\varepsilon z)dz
\end{equation}
and called the \textit{Steklov smoothing operator}. Note that
$$
\|S_\varepsilon u\|_{L^2(\mathbb{R}^d)} \le \|u\|_{L^2(\mathbb{R}^d)}.
$$ 
Obviously, $D^{\alpha}S_\varepsilon u=S_\varepsilon D^\alpha u$ for $u\in H^s(\mathbb{R}^d)$ and any 
multi-index $\alpha$ such that $|\alpha|\le s$. Therefore,
$$
\|S_\varepsilon u\|_{H^s(\mathbb{R}^d)} \le \|u\|_{H^s(\mathbb{R}^d)}.
$$
The following are a few properties of Steklov's operator; see \cite{Suslina1301,Suslina1302}.

\begin{prop}\label{prop2.1}
For any $u\in H^1(\mathbb{R}^d)$ we have
$$
\|S_\varepsilon u-u\|_{L^2(\mathbb{R}^d)} \le C\varepsilon \|\nabla u\|_{L^2(\mathbb{R}^d)},
$$
where $C$ depends only on $d$.
\end{prop}

We will use  the notation $f^\varepsilon(x)=f(x/\varepsilon)$.

\begin{prop}\label{prop2.2}
Let $f(x)$ be a $1$-periodic function in $\mathbb{R}^d$ such that $f\in L^2(Y)$.  Then for any $u\in L^2(\mathbb{R}^d)$,
$$
\|f^\varepsilon S_\varepsilon u\|_{L^2(\mathbb{R}^d)} \le \|f\|_{L^2(Y)}\|u\|_{L^2(\mathbb{R}^d)}.
$$
\end{prop}


\subsection{Homogenization of Stokes systems}

We refer the reader to \cite{Lions,GuShen15} for details of weak solutions and homogenization theory of Stokes system. 

Let $\Omega$ be a bounded Lipschitz domain in $\mathbb{R}^d$. For $u,v \in H^1(\Omega;\mathbb{R}^d)$, we define the bilinear form $a_\varepsilon(\cdot,\cdot)$ by
$$
a_\varepsilon(u,v)=\int_\Omega a_{ij}^{\alpha\beta}\left(\frac{x}{\varepsilon}\right)\frac{\partial u^\beta}{\partial x_j}\frac{\partial v^\alpha}{\partial x_i} dx.
$$
For $F\in H^{-1}(\Omega;\mathbb{R}^d)$ and $g\in L^2(\Omega)$, we say that $(u_\varepsilon,p_\varepsilon)\in H^1(\Omega;\mathbb{R}^d)\times L^2(\Omega)$ is a weak solution of the following Stokes system in $\Omega$,
\begin{equation}\label{Stokes}
\left\{
\begin{aligned}
\mathcal{L}_\varepsilon(u_\varepsilon)+\nabla p_\varepsilon &= F \\
\text{ div }u_\varepsilon &=g, \\
\end{aligned}
\right.
\end{equation}
if for any $\varphi \in C_0^1(\Omega;\mathbb{R}^d)$,
$$
a_\varepsilon(u_\varepsilon,\varphi)-\int_\Omega p_\varepsilon \text{ div}(\varphi)=\langle F,\varphi\rangle
$$
and $\text{div}(u_\varepsilon)=g$ in $\Omega$ (in the sense of distribution).

\begin{theorem}\label{weak-solution-theorem}
Let $\Omega$ be a bounded Lipschitz domain in $\mathbb{R}^d$. Suppose $A(y)$ satisfies the ellipticity condition (\ref{Ellipticity}). Let $F\in H^{-1}(\Omega;\mathbb{R}^d)$, $g\in L^2(\Omega)$ and $f\in H^{1/2}(\partial\Omega;\mathbb{R}^d)$ satisfy the compatibility condition (\ref{Compatibility}). Then there exist a unique $u_\varepsilon \in H^1(\Omega;\mathbb{R}^d)$ and $p_\varepsilon \in L^2(\Omega)$, unique up to constants, such that $(u_\varepsilon,p_\varepsilon)$ is a weak solution of (\ref{Stokes}) and $u_\varepsilon=f$ on $\partial\Omega$.
Moreover,
\begin{equation}\label{energy}
\|u_\varepsilon\|_{H^1(\Omega)}+\|p_\varepsilon-\average_\Omega p_\varepsilon\|_{L^2(\Omega)} 
\le C\Big\{\|F\|_{H^{-1}(\Omega)}+\|g\|_{L^2(\Omega)}+\|f\|_{H^{1/2}(\partial\Omega)}\Big\},
\end{equation}
where $C$ depends only on $d$, $\mu$, and $\Omega$.
\end{theorem}

Theorem \ref{weak-solution-theorem} is proved by using the Lax-Milgram Theorem.
We mention that if $\Omega$ is $C^{1,1}$ and $A$ is a constant matrix, the weak solution $(u, p)$,
given by Theorem \ref{weak-solution-theorem}, is in $H^2(\Omega; \mathbb{R}^d)\times H^1(\Omega)$, 
provided that $F\in L^2(\Omega; \mathbb{R}^d)$, $g\in H^1(\Omega)$ and $f\in H^{3/2}(\partial\Omega; \mathbb{R}^d)$.
Moreover, 
\begin{equation}\label{reg}
\|u\|_{H^2(\Omega)}+\|\nabla  p\|_{L^2(\Omega)} 
\le C\Big\{\|F\|_{L^2(\Omega)}+\|g\|_{H^1(\Omega)}+\|f\|_{H^{3/2}(\partial\Omega)}\Big\},
\end{equation}
where $C$ depends only on $d$, $\mu$, and $\Omega$ (see e.g. \cite{GiaquintaModica82}).

We denote by $H_{\text{per}}^1(Y;\mathbb{R}^d)$ the closure in $H^1(Y;\mathbb{R}^d)$ of 
$C_{\text{per}}^\infty(Y;\mathbb{R}^d)$, the set of $C^\infty$ 1-periodic and $\mathbb{R}^d$-valued functions 
in $\mathbb{R}^d$. Let
$$
a_{\text{per}}(\psi,\phi)=\int_Y a_{ij}^{\alpha\beta}(y)\frac{\partial \psi^\beta}{\partial x_j}\frac{\partial \phi^\alpha}{\partial x_i} dy,
$$
where $\psi,\phi \in H^1_{\text{per}}(Y;\mathbb{R}^d)$. 
Define
$$
V_{\text{per}}(Y)=\left\{u\in H_{\text{per}}^1(Y;\mathbb{R}^d): \text{ div}(u)=0 \text{ in }Y \text{ and }\int_Y u=0 \right\}.
$$
By applying the Lax-Milgram Theorem to $a_{\text{per}} (\psi, \phi)$ on
the Hilbert space $V_{\text{per}} (Y)$, one may show that
for each $1\le j,\beta \le d$, there exist 1-periodic functions $(\chi_j^\beta,\pi_j^\beta)\in H^1_{\text{loc}}(\mathbb{R}^d;\mathbb{R}^d)\times L^2_{\text{loc}}(\mathbb{R}^d)$, which are called the correctors for the Stokes system (\ref{Stokes}), such that
\begin{equation}\label{Corrector}
\left\{
\begin{aligned}
\mathcal{L}_1(\chi_j^\beta+P_j^\beta)+\nabla \pi_j^\beta &= 0 \quad \text{in }\mathbb{R}^d, \\
\text{ div }\chi_j^\beta &=0 \quad \text{in }\mathbb{R}^d,\\
\int_Y \pi_j^\beta=0, \int_Y \chi_j^\beta &=0,\\
\end{aligned}
\right.
\end{equation}
where $P_j^\beta=P_j^\beta(y)=y_j e^\beta=y_j(0,\cdots,1,\cdots,0)$ with 1 in the $\beta^{\text{th}}$ position. Note that
$$
\|\chi_j^\beta\|_{H^1(Y)}+\|\pi_j^\beta\|_{L^2(Y)} \le C,
$$
where $C$ depends only on $d$ and $\mu$. The homogenized system for the Stokes system (\ref{Stokes}) is given by
\begin{equation}\label{Stokes0}
\left\{
\begin{aligned}
\mathcal{L}_0(u_0)+\nabla p_0 &= F \\
\text{ div }u_0 &=g, \\
\end{aligned}
\right.
\end{equation}
where $\mathcal{L}_0=-\text{div}(\widehat{A}\nabla)$ is a second-order elliptic operator with constant coefficients, and $\widehat{A}=(\widehat{a}_{ij}^{\alpha\beta})$, with
\begin{equation*}
\widehat{a}_{ij}^{\alpha\beta}=a_{\text{per}}(\chi_j^\beta+P_j^\beta,\chi_i^\alpha+P_i^\alpha).
\end{equation*}
We  remark that $(\widehat{A})^*=\widehat{A^*}$, and the effective matrix $\widehat{A}$ satisfies the ellipticity condition $\mu|\xi|^2\le \widehat{a}_{ij}^{\alpha\beta}\xi_i^\alpha\xi_j^\beta \le \mu_1|\xi|^2$, for any $\xi\in \mathbb{R}^{d\times d}$ and $\mu_1$ depends only on $d$ and $\mu$. The following is a homogenization theorem for the Stokes system.

\begin{theorem}
Suppose that $A(y)$ satisfies ellipticity condition (\ref{Ellipticity}) and periodicity condition (\ref{Periodicity}). Let $\Omega$ be a bounded Lipschitz domain. Let $(u_\varepsilon,p_\varepsilon)\in H^1(\Omega;\mathbb{R}^d)\times L^2(\Omega)$ be a weak solution of (\ref{DirichletStokes}), where $F\in H^{-1}(\Omega;\mathbb{R}^d)$, $g\in L^2(\Omega)$ and $f\in H^{1/2}(\partial\Omega;\mathbb{R}^d)$. Assume that $\int_\Omega p_\varepsilon=0$. Then, as $\varepsilon \rightarrow 0$,
$$
\left\{
\begin{aligned}
u_\varepsilon &\rightarrow u_0 \quad\text{\rm strongly in }L^2(\Omega;\mathbb{R}^d),\\
u_\varepsilon &\rightharpoonup u_0 \quad\text{\rm weakly in }H^1(\Omega;\mathbb{R}^d),\\
p_\varepsilon &\rightharpoonup p_0 \quad\text{\rm weakly in }L^2(\Omega),\\
A(x/\varepsilon)\nabla u_\varepsilon &\rightharpoonup \widehat{A}\nabla u_0 \quad\text{\rm weakly in }L^2(\Omega;\mathbb{R}^{d\times d}).\\
\end{aligned}
\right.
$$
Moreover, $(u_0,p_0)$ is the weak solution of the homogenized problem (\ref{DirichletStokes0}).
\end{theorem}



\section{Convergence rates for $u_\varepsilon$ in $H^1$}
\setcounter{equation}{0}
From now on we will assume that
 $\Omega$ is a bounded domain with boundary of class $C^{1,1}$, $F\in L^2(\Omega; \mathbb{R}^d)$,
 $g\in H^1(\Omega)$, and $f\in H^{3/2}(\partial\Omega; \mathbb{R}^d)$.
  We fix a linear continuous extension operator 
\begin{equation*}
E_\Omega: H^2(\Omega;\mathbb{R}^d) \rightarrow H^2(\mathbb{R}^d;\mathbb{R}^d),
\end{equation*}
and let
\begin{equation}\label{defextension}
\widetilde{u}_0=E_\Omega u_0,
\end{equation} 
so that $\widetilde{u}_0 =u_0$ in $\Omega$ and
\begin{equation}\label{extension}
\|\widetilde{u}_0\|_{H^2(\mathbb{R}^d)} \le C\|u_0\|_{H^2(\Omega)},
\end{equation}
where $C$ depends on $\Omega$.
We introduce a first order approximation of $u_\varepsilon$,
$$
v_\varepsilon= u_0+\varepsilon \chi^\varepsilon S_\varepsilon (\nabla \widetilde{u}_0).
$$
Let $(w_\varepsilon,\tau_\varepsilon) \in H^1(\Omega;\mathbb{R}^d)\times L^2(\Omega)$  be a weak solution of
\begin{equation}\label{w}
\left\{
\aligned
\mathcal{L}_\varepsilon(w_\varepsilon)+\nabla \tau_\varepsilon&=0  &\quad  &\text{ in }\Omega,\\
\text{ div}(w_\varepsilon)&= \varepsilon \text{ div}\big(\chi^\varepsilon S_\varepsilon \nabla \widetilde{u}_0\big)&\quad &\text{ in }\Omega,\\
w_\varepsilon&=\varepsilon \chi^\varepsilon S_\varepsilon (\nabla \widetilde{u}_0)
&\quad &\text{ on }\partial\Omega.\\
\endaligned
\right.
\end{equation}
We will use $w_\varepsilon$ to approximate the difference between $u_\varepsilon$ and its first order approximation $v_\varepsilon$.  
To this end, for $1\le i,j,\alpha,\beta \le d$, we let
\begin{equation}\label{b}
b_{ij}^{\alpha\beta}(y)=a_{ij}^{\alpha\beta}(y)+
a_{ik}^{\alpha\gamma}(y)\frac{\partial}{\partial y_k}\big(\chi_{j}^{\gamma\beta}\big)-\widehat{a}_{ij}^{\alpha\beta}.
\end{equation}
Note that $b_{ij}^{\alpha\beta}$ is $1$-periodic. By the definition of $\chi$ and $\widehat{A}$, $b_{ij}^{\alpha\beta}\in L^2(Y)$ satisfies
$$
\int_Y b_{ij}^{\alpha\beta}(y)\, dy=0.
$$
and, for each $1\le \alpha, \beta, j\le d$,
\begin{equation}\label{b1}
\aligned
\frac{\partial}{\partial y_i}\big(b_{ij}^{\alpha\beta}(y)\big)
&=\frac{\partial}{\partial y_i}\big(a_{ij}^{\alpha\beta}(y)\big)
+\frac{\partial}{\partial y_i}\left(a_{ik}^{\alpha\gamma}(y)\frac{\partial \chi_{j}^{\gamma\beta}}{\partial y_k}\right)\\
&=\frac{\partial}{\partial y_i}\big(a_{ij}^{\alpha\beta}(y)\big)-\frac{\partial}{\partial y_i}\left(a_{ik}^{\alpha\gamma}(y)\frac{\partial P_{j}^{\gamma\beta}}{\partial y_k}\right)+\frac{\partial}{\partial y_\alpha}(\pi_j^\beta) \\
&=\frac{\partial}{\partial y_\alpha}(\pi_j^\beta).
\endaligned
\end{equation}

\begin{lemma}\label{lemma3.1}
There exist $\varPhi_{kij}^{\alpha\beta} \in H_{\text{\rm per}}^1(Y)$ and $q_{ij}^\beta\in H_{\text{\rm per}}^1(Y)$ such that
\begin{equation}\label{elemma3.1}
b_{ij}^{\alpha\beta}=\frac{\partial}{\partial y_k}(\Phi_{kij}^{\alpha\beta})+\frac{\partial}{\partial y_\alpha}(q_{ij}^\beta) \quad \text{\rm and }\quad \Phi_{kij}^{\alpha\beta}=-\Phi_{ikj}^{\alpha\beta}.
\end{equation}
Moreover,
\begin{equation}\label{elemma3.1.a}
\|\Phi_{kij}^{\alpha\beta}\|_{L^2(Y)}+\|q_{ij}^{\beta}\|_{L^2(Y)} \le C,
\end{equation}
where $C$ depends only on $d$ and $\mu$.
\end{lemma}

\begin{proof}
Fix $1\le i,j,\beta\le d$.
 There exist $f_{ij}^{\beta}=(f_{ij}^{\alpha\beta})\in H_{\text{per}}^2(Y;\mathbb{R}^d)$ and $q_{ij}^{\beta}\in H^1_{\text{per}}(Y)$ satisfying the following Stokes system,
\begin{equation}\label{f}
\left\{
\begin{aligned}
\Delta f_{ij}^{\beta} +\nabla q_{ij}^\beta &= b_{ij}^{\beta} &\quad\text{ in }Y,\\
\text{div} (f_{ij}^\beta) &=0 &\quad\text{ in }Y,\\
\int_Y f_{ij}^{\beta\ } dy& =0,
\end{aligned}
\right.
\end{equation}
where $b_{ij}^\beta =(b_{ij}^{\alpha\beta})$.
We now define
$$
\Phi_{kij}^{\alpha\beta}(y)=\frac{\partial}{\partial y_k}(f_{ij}^{\alpha\beta})-\frac{\partial}{\partial y_i}(f_{kj}^{\alpha\beta}).
$$
Clearly, $\Phi_{kij}^{\alpha\beta}\in H^1_{\text{per}}(Y)$ and $\Phi_{kij}^{\alpha\beta}=-\Phi_{ikj}^{\alpha\beta}$. 
Note that, by (\ref{b1}) and (\ref{f}), $\frac{\partial f_{ij}^{\alpha\beta}}{\partial y_i}
\in H^1_{\text{per}} (Y)$ satisfies
\begin{equation}\label{e3.1}
\left\{
\aligned
\Delta\left(\frac{\partial f_{ij}^{\alpha\beta}}{\partial y_i}\right)& 
=-\frac{\partial}{\partial y_\alpha}\left(\frac{\partial q_{ij}^{\beta}}{\partial y_i}\right)+\frac{\partial b_{ij}^{\alpha\beta}}{\partial y_i}
=\frac{\partial}{\partial y_\alpha}\left(\pi_j^\beta-\frac{\partial q_{ij}^\beta}{\partial y_i} \right),\\
\frac{\partial}{\partial y_\alpha}\left(\frac{\partial f_{ij}^{\alpha\beta}}{\partial y_i}\right)&=0.
\endaligned
\right.
\end{equation}
It follows by the energy estimates that $\frac{\partial f_{ij}^{\alpha\beta}}{\partial y_i}$ is  constant.
Hence,
$$
\frac{\partial}{\partial y_k}(\Phi_{kij}^{\alpha\beta})=\frac{\partial^2}{\partial y_k\partial y_k}(f_{ij}^{\alpha\beta})-\frac{\partial}{\partial y_i}\left(\frac{\partial}{\partial y_k}(f_{kj}^{\alpha\beta})\right)= b_{ij}^{\alpha\beta}-\frac{\partial}{\partial y_\alpha}(q_{ij}^{\beta}).
$$
Furthermore, since $\|\chi_j^\beta\|_{H^1(Y)}\le C$, then
\begin{equation*}
\|\Phi_{kij}^{\alpha\beta}\|_{L^2(Y)}+\|q_{ij}^{\beta}\|_{L^2(Y)} \le C\|b_{ij}^{\alpha\beta}\|_{L^2(Y)}\le C,
\end{equation*}
where $C$ depends only on $d$ and $\mu$.
This completes the proof.
\end{proof}
\begin{remark}\label{remark3.1}
Recall that $\pi_j^\beta$ and $q_{ij}^\beta$ are both 1-periodic. By (\ref{e3.1})
and the fact that  $\frac{\partial f_{ij}^{\alpha\beta}}{\partial y_i}$ is  constant,
we see that  $\pi_j^\beta$ and $\frac{\partial q_{ij}^\beta}{\partial y_i}$ differ only by a constant. 
Since $\int_Y \pi_j^\beta=0$, we obtain  the following relation,
\begin{equation}\label{qpi}
\pi_j^\beta=\frac{\partial q_{ij}^\beta}{\partial y_i}.
\end{equation}
\end{remark}

\begin{lemma}\label{lemma3.2} 
Let $\Omega$ be a bounded $C^{1,1}$ domain. Suppose that $A$ satisfies ellipticity condition (\ref{Ellipticity}) and
 periodicity condition (\ref{Periodicity}). Given $g\in H^1(\Omega)$ and $f\in H^{3/2}(\partial\Omega;\mathbb{R}^d)$ 
 satisfying the compatibility condition (\ref{Compatibility}),  for $F\in L^2(\Omega;\mathbb{R}^d)$, 
 let $(u_\varepsilon,p_\varepsilon)$, $(u_0,p_0)$ and $(w_\varepsilon,\tau_\varepsilon)$ 
 be weak solutions of Dirichlet problems (\ref{DirichletStokes}), (\ref{DirichletStokes0}) and (\ref{w}), respectively. Then,
\begin{equation}\label{elemma3.2}
\|u_\varepsilon-u_0-\varepsilon \chi^\varepsilon S_\varepsilon (\nabla \widetilde{u}_0)+w_\varepsilon\|_{H_0^1(\Omega)}
 \le C\varepsilon\|u_0\|_{H^2(\Omega)},
\end{equation}
where $C$ depends only on $d$, $\mu$, and $\Omega$.
\end{lemma}

\begin{proof}
Let
$$
z_\varepsilon=u_\varepsilon-u_0-\varepsilon \chi^\varepsilon S_\varepsilon(\nabla \widetilde{u}_0)+w_\varepsilon.
$$
Then
$$ 
\text{div}(z_\varepsilon)=0 \quad\text{ in } \Omega \quad\text{ and }\quad z_\varepsilon=0 \quad\text{ on } \partial\Omega.
$$
Now we compute $\mathcal{L}_\varepsilon(z_\varepsilon)$,
$$
\begin{aligned}
&(\mathcal{L}_\varepsilon(z_\varepsilon))^\alpha
= -\frac{\partial [p_\varepsilon-p_0+\tau_\varepsilon]}{\partial x_\alpha}
-\frac{\partial}{\partial x_i}\left(\Big[\widehat{a}_{ij}^{\alpha\beta}-a_{ij}^{\alpha\beta}(x/\varepsilon)\Big]\frac{\partial u_0^\beta}{\partial x_j}\right)\\
&\quad+\frac{\partial}{\partial x_i}\left(a_{ik}^{\alpha\gamma}(x/\varepsilon)\frac{\partial}{\partial x_k}
\Big[\varepsilon \chi_j^{\gamma\beta}(x/\varepsilon)\Big]S_\varepsilon\frac{\partial \widetilde{u}_0^\beta}{\partial x_j}\right) + \varepsilon\frac{\partial}{\partial x_i}\left(a_{ik}^{\alpha\gamma}(x/\varepsilon)\chi_j^{\gamma\beta}(x/\varepsilon)S_\varepsilon \frac{\partial^2 \widetilde{u}_0^\beta}{\partial x_k \partial x_j}\right)\\
&= -\frac{\partial [p_\varepsilon-p_0+\tau_\varepsilon]}{\partial x_\alpha}- \frac{\partial}{\partial x_i}\left(\Big[\widehat{a}_{ij}^{\alpha\beta}-a_{ij}^{\alpha\beta}(x/\varepsilon)\Big]\bigg[\frac{\partial u_0^\beta}{\partial x_j}-S_\varepsilon\frac{\partial \widetilde{u}_0^\beta}{\partial x_j}\bigg]\right)\\
&\quad +\frac{\partial}{\partial x_i}\left(b_{ij}^{\alpha\beta}(x/\varepsilon)S_\varepsilon\frac{\partial \widetilde{u}_0^\beta}{\partial x_j}\right)+\varepsilon\frac{\partial}{\partial x_i}\left(a_{ik}^{\alpha\gamma}(x/\varepsilon)\chi_j^{\gamma\beta}(x/\varepsilon)S_\varepsilon \frac{\partial^2 \widetilde{u}_0^\beta}{\partial x_k \partial x_j}\right).\\
\end{aligned}
$$
Using Lemma \ref{lemma3.1}, we may write
\begin{equation}\label{e3.2}
\aligned
\frac{\partial}{\partial x_i}\left(b_{ij}^{\alpha\beta}(x/\varepsilon)S_\varepsilon\frac{\partial \tilde{u}_0^\beta}{\partial x_j}\right)&=\frac{\partial}{\partial x_i}\left(\bigg[\frac{\partial}{\partial x_k}\Big(\varepsilon \Phi_{kij}^{\alpha\beta}(x/\varepsilon)\Big)+\frac{\partial}{\partial x_\alpha}\Big(\varepsilon q_{ij}^\beta(x/\varepsilon)\Big)\bigg]S_\varepsilon\frac{\partial \widetilde{u}_0^\beta}{\partial x_j}\right)\\
&= I_1+I_2.\\
\endaligned
\end{equation}
Since $\Phi_{kij}^{\alpha\beta}=-\Phi_{ikj}^{\alpha\beta}$, we see that
$$
\begin{aligned}
I_1&=\frac{\partial^2}{\partial x_i \partial x_k}\left(\varepsilon \Phi_{kij}^{\alpha\beta}(x/\varepsilon)S_\varepsilon\frac{\partial 
\widetilde{u_0}^\beta}{\partial x_j}\right)-\varepsilon\frac{\partial}{\partial x_i}\left(\Phi_{kij}^{\alpha\beta}(x/\varepsilon)S_\varepsilon\frac{\partial^2 \widetilde{u}_0^\beta}{\partial x_j\partial x_k}\right)\\
&=-\varepsilon\frac{\partial}{\partial x_i}\left(\Phi_{kij}^{\alpha\beta}(x/\varepsilon)S_\varepsilon\frac{\partial^2 \widetilde{u}_0^\beta}{\partial x_j\partial x_k}\right).
\end{aligned}
$$
For the second term in the RHS of (\ref{e3.2}), we have
\begin{equation}\label{e3.5}
\begin{aligned}
I_2&= \frac{\partial}{\partial x_\alpha}\left(\frac{\partial}{\partial x_i}\bigg[\varepsilon q_{ij}^\beta(x/\varepsilon)S_\varepsilon\frac{\partial \widetilde{u}_0^\beta}{\partial x_j}\bigg]\right)-\frac{\partial}{\partial x_i}\left(\varepsilon q_{ij}^\beta(x/\varepsilon)S_\varepsilon\frac{\partial^2 \widetilde{u}_0^\beta}{\partial x_\alpha \partial x_j}\right)\\
&=I_3-\frac{\partial}{\partial x_i}\left(\varepsilon q_{ij}^\beta(x/\varepsilon)S_\varepsilon\frac{\partial^2 \widetilde{u}_0^\beta}{\partial x_\alpha \partial x_j}\right).
\end{aligned}
\end{equation}
In view of (\ref{qpi}), for the first term on the RHS of (\ref{e3.5}), we obtain 
\begin{equation}
 I_3 
 = \frac{\partial}{\partial x_\alpha}\left(\pi_j^\beta(x/\varepsilon) S_\varepsilon\frac{\partial \widetilde{u}_0^\beta}{\partial x_j}\right)+\frac{\partial}{\partial x_\alpha}\left(\varepsilon q_{ij}^\beta(x/\varepsilon)S_\varepsilon\frac{\partial^2 \widetilde{u}_0^\beta}{\partial x_j \partial x_i}\right).
\end{equation}

Putting altogether, we have shown that
\begin{equation}\label{e3.6}
\begin{aligned}
&\left(\mathcal{L}_\varepsilon(z_\varepsilon)\right)^\alpha+\frac{\partial}{\partial x_\alpha}\left(p_\varepsilon-p_0-\pi_j^\beta(x/\varepsilon) S_\varepsilon\frac{\partial \widetilde{u}_0^\beta}{\partial x_j}-\varepsilon q_{ij}^\beta(x/\varepsilon)S_\varepsilon\frac{\partial^2 \widetilde{u}_0^\beta}{\partial x_j \partial x_i}+\tau_\varepsilon\right)\\
&=\varepsilon\frac{\partial}{\partial x_i}\left(\Big[a_{ij}^{\alpha\gamma}(x/\varepsilon)\chi_k^{\gamma\beta}(x/\varepsilon)-\Phi_{kij}^{\alpha\beta}(x/\varepsilon)\Big]S_\varepsilon\frac{\partial^2 \widetilde{u}_0^\beta}{\partial x_j \partial x_k}\right)\\
&\qquad \qquad 
-\varepsilon\frac{\partial}{\partial x_i}\left(q_{ij}^\beta(x/\varepsilon)S_\varepsilon\frac{\partial^2 \widetilde{u}_0^\beta}{\partial x_\alpha \partial x_j}\right)\\ 
&\qquad\qquad
-\frac{\partial}{\partial x_i}\left(\Big[\widehat{a}_{ij}^{\alpha\beta}-a_{ij}^{\alpha\beta}(x/\varepsilon)\Big]\bigg[\frac{\partial u_0^\beta}{\partial x_j}-S_\varepsilon\frac{\partial \widetilde{u}_0^\beta}{\partial x_j}\bigg]\right).\\
\end{aligned}
\end{equation}
Since  $z_\varepsilon \in H^1_0(\Omega; \mathbb{R}^d)$ and $\text{div} (z_\varepsilon)=0$ in $\Omega$, 
it follows from (\ref{e3.6}) by the energy estimate (\ref{energy}) that
$$
\aligned
c \int_\Omega |\nabla z_\varepsilon|^2 dx
&\le \varepsilon^2\int_\Omega \Big|\big[ |\chi(x/\varepsilon)|+|\Phi (x/\varepsilon)|\big]S_\varepsilon (\nabla^2 \widetilde{u}_0)\Big|^2\, dx\\
&\quad +\varepsilon^2\int_\Omega \Big| q (x/\varepsilon)S_\varepsilon (\nabla^2 \widetilde{u}_0)\Big|^2\, dx
+ \int_\Omega \Big|\nabla u_0-S_\varepsilon (\nabla  \widetilde{u}_0) \Big|^2\, dx.
\endaligned
$$
Now  we apply Propositions \ref{prop2.1}-\ref{prop2.2} as well as (\ref{extension}). This gives
$$
\begin{aligned}
\|\nabla z_\varepsilon\|_{L^2(\Omega)} 
&\le  C\varepsilon\left(\|\chi\|_{L^2(Y)}+\|\Phi\|_{L^2(Y)}+\|q\|_{L^2(Y)}+1\right)\|\nabla^2\widetilde{u}_0\|_{L^2(\mathbb{R}^d)}\\
&\le  C\varepsilon\|\nabla^2\widetilde{u}_0\|_{L^2(\mathbb{R}^d)} \\
&\le C\varepsilon\|u_0\|_{H^2(\Omega)},
\end{aligned}
$$
where $C$ depends only on $d$, $\mu$ and $\Omega$. 
Hence we have proved the desired result,
$
\|z_\varepsilon\|_{H_0^1(\Omega)} \le C\varepsilon\|u_0\|_{H^2(\Omega)},
$
and completed the proof.
\end{proof}

For $r>0$, let
 $$
 \aligned
 (\partial\Omega)_r &=\{x\in\mathbb{R}^d: \text{dist}(x,\partial\Omega)\le r\},\\
 \Omega_r & =\{x\in\Omega: \text{dist}(x,\partial\Omega)\le r\}.
 \endaligned
 $$
  We choose two cut-off functions $\theta_\varepsilon(x)$ and $\widetilde{\theta}_\varepsilon(x)$ in $\mathbb{R}^d$ 
  satisfying  the following conditions,
\begin{equation}\label{cutoff1}
\begin{aligned}
& \theta_\varepsilon\in C_0^\infty(\mathbb{R}^d), \quad \text{supp} (\theta_\varepsilon) \subset (\partial \Omega)_\varepsilon,\quad 0\le \theta_\varepsilon(x)\le 1,\\
& \theta_\varepsilon|_{\partial\Omega}=1, \quad |\nabla \theta_\varepsilon|\le \kappa/\varepsilon,
\end{aligned}
\end{equation}
and
\begin{equation}\label{cutoff2}
\begin{aligned}
& \widetilde{\theta}_\varepsilon\in C_0^\infty(\mathbb{R}^d), \quad \text{supp} (\widetilde{\theta}_\varepsilon) \subset (\partial \Omega)_{2\varepsilon},\quad 0\le \widetilde{\theta}_\varepsilon(x)\le 1,\\
& \widetilde{\theta}_\varepsilon(x)=1 \text{ for } x\in (\partial\Omega)_\varepsilon, 
\quad |\nabla \widetilde{\theta}_\varepsilon |\le \widetilde{\kappa}/\varepsilon.
\end{aligned}
\end{equation}

The following  is an estimate for integrals near the boundary, see \cite{Suslina1302} for example.

\begin{lemma}\label{lemma4.0}
Let $\Omega\subset \mathbb{R}^d$ be a bounded $C^1$ domain.
 Then, for any function $u\in H^1(\Omega)$,
$$
\int_{\Omega_r} |u|^2 dx \le C r \|u\|_{H^1(\Omega)}\|u\|_{L^2(\Omega)}.
$$
Moreover, for any 1-periodic function $f \in L^2(Y)$ and $u\in H^1(\mathbb{R}^d)$,
$$
\int_{(\partial\Omega)_{2\varepsilon}}|f^\varepsilon|^2|S_\varepsilon u|^2 dx \le C\varepsilon\|f\|_{L^2(Y)}
\|u\|_{H^1(\mathbb{R}^d)}\|u\|_{L^2(\mathbb{R}^d)},
$$
where $C$ depends only on $\Omega$.
\end{lemma}

We are now ready to give the proof of (\ref{ethm1.1a}).

\begin{proof}[\textbf{Proof of estimate (\ref{ethm1.1a})}]
By Lemma \ref{lemma3.2}, the problem has been reduced to estimating  $w_\varepsilon$ in $H^1$.
Notice that by the energy estimate (\ref{energy}),
\begin{equation}\label{estimate-w}
\aligned
&\|w_\varepsilon\|_{H^1(\Omega)} 
\le C\varepsilon\|\chi^\varepsilon S_\varepsilon \nabla \widetilde{u}_0\|_{H^{1/2}(\partial \Omega)} 
+ C \varepsilon\| \text{div} \big(\chi^\varepsilon S_\varepsilon\nabla \widetilde{u}_0) \|_{L^2(\Omega)} \\
&\le C\varepsilon\|\theta_\varepsilon \chi^\varepsilon S_\varepsilon \nabla \widetilde{u}_0\|_{H^1(\Omega)}
+C\varepsilon \| \chi^\varepsilon \nabla S_\varepsilon  (\nabla \widetilde{u}_0)\|_{L^2(\Omega)}\\
&\le C \varepsilon
\Big\{ \| \chi^\varepsilon S_\varepsilon (\nabla \widetilde{u}_0) \|_{L^2(\Omega)}
+\| (\nabla \theta_\varepsilon) \chi^\varepsilon S_\varepsilon (\nabla \widetilde{u}_0)\|_{L^2(\Omega)}\\
&\qquad \qquad +\varepsilon^{-1} \| \theta_\varepsilon (\nabla \chi)^\varepsilon S_\varepsilon (\nabla \widetilde{u}_0) \|_{L^2(\Omega)}
+\|  \chi^\varepsilon S_\varepsilon (\nabla^2 \widetilde{u}_0) \|_{L^2(\Omega)} \Big\}\\
&\le C \varepsilon \Big\{ \| \widetilde{u}_0\|_{H^2(\mathbb{R}^d)}
+\varepsilon^{-1} \|  \chi^\varepsilon S_\varepsilon (\nabla \widetilde{u}_0) \|_{L^2(\Omega_\varepsilon)}
+\varepsilon^{-1} \|  (\nabla \chi)^\varepsilon S_\varepsilon (\nabla \widetilde{u}_0) \|_{L^2(\Omega_\varepsilon)}\Big\}\\
&\le C \varepsilon^{1/2}  \| \widetilde{u}_0\|_{H^2(\mathbb{R}^d)},
\endaligned
\end{equation}
where we have used Proposition \ref{prop2.2} for the fourth inequality and Lemma \ref{lemma4.0}
for the last.
We point out that the fact $\text{div} (\chi)=0$ in $\mathbb{R}^d$ is also used for the second inequality in (\ref{estimate-w}).
Therefore,
$$
\aligned
\|u_\varepsilon-u_0-\varepsilon\chi^\varepsilon S_\varepsilon\left(\nabla \widetilde{u}_0\right)\|_{H^1(\Omega)} 
&\le \|z_\varepsilon\|_{H^1(\Omega)}+\|w_\varepsilon\|_{H^1(\Omega)}\\
& \le C\sqrt{\varepsilon}\|u_0\|_{H^2(\Omega)},
\endaligned
$$
where $C$ depends only on $d$, $\mu$, and $\Omega$.
This completes the proof.
\end{proof}



\section{Convergence rates for the pressure term}
\setcounter{equation}{0}

To prove estimate(\ref{ethm1.1b}), we first recall that if $(u_\varepsilon,p_\varepsilon)\in H^1(\Omega;\mathbb{R}^d)\times L^2(\Omega)$ is a weak solution of the Stokes system (\ref{DirichletStokes}), then 
\begin{equation}\label{ep}
\|p_\varepsilon-\average_\Omega p_\varepsilon\|_{L^2(\Omega)}\le C \|\nabla p_\varepsilon\|_{H^{-1}(\Omega)}\le 
C \Big\{ \|F\|_{H^{-1}(\Omega)}+ \|u_\varepsilon\|_{H^1(\Omega)}\Big\},
\end{equation}
where $C$ depends only on $d$, $\mu$, and $\Omega$
(see e.g. \cite{Temam77}).

\begin{proof}[\textbf{Proof of estimate (\ref{ethm1.1b})}]
Since $\int_\Omega p_\varepsilon=\int_\Omega p_0=0$, using (\ref{ep}) and (\ref{e3.6}), we see that
\begin{equation}\label{e3.9}
\begin{aligned}
&\|p_\varepsilon-p_0-\Big[\big(\pi^\varepsilon S_\varepsilon\nabla\widetilde{u}_0 +\varepsilon q^\varepsilon S_\varepsilon\nabla^2\widetilde{u}_0-\tau_\varepsilon\big)-\average_\Omega \big(\pi^\varepsilon S_\varepsilon\nabla\widetilde{u}_0 +\varepsilon q^\varepsilon S_\varepsilon\nabla^2\widetilde{u}_0-\tau_\varepsilon\big)\Big]\|_{L^2(\Omega)}\\
&\le C\|\nabla \Big[p_\varepsilon-p_0-\pi^\varepsilon S_\varepsilon\left(\nabla\tilde{u}_0 \right)-\varepsilon q^\varepsilon S_\varepsilon\left(\nabla^2\tilde{u}_0\right)+\tau_\varepsilon\Big]\|_{H^{-1}(\Omega)}\\
&\le C\left\{\|\nabla z_\varepsilon\|_{L^2(\Omega)}+\varepsilon\left\| 
\big(|\chi^\varepsilon|+|\varPhi^\varepsilon|+|q^\varepsilon|\big)S_\varepsilon\left(\nabla^2\widetilde{u}_0\right)\right\|_{L^2(\Omega)}+\|S_\varepsilon\left(\nabla\widetilde{u}_0\right)-\nabla u_0\|_{L^2(\Omega)}\right\}\\
&\le C\varepsilon\|u_0\|_{H^2(\Omega)},
\end{aligned}
\end{equation}
where the last inequality follows from the proof of Lemma \ref{lemma3.2}.
Note that  by Propostion $\ref{prop2.2}$ and (\ref{extension}), 
\begin{equation}\label{e3.10}
\varepsilon\|q^\varepsilon S_\varepsilon\nabla^2 \widetilde{u}_0
-\average_\Omega q^\varepsilon S_\varepsilon\nabla^2 \tilde{u}_0\|_{L^2(\Omega)} 
\le C\varepsilon\|\widetilde{u}_0\|_{H^2(\mathbb{R}^d)} \le C\varepsilon\|u_0\|_{H^2(\Omega)}.
\end{equation}
Also, by the definition of $(w_\varepsilon,\tau_\varepsilon)$ and (\ref{ep}),
\begin{equation}\label{e3.11}
\|\tau_\varepsilon-\average_\Omega \tau_\varepsilon\|_{L^2(\Omega)} \le C\|\nabla \tau_\varepsilon\|_{H^{-1}(\Omega)}\le C\|\nabla w_\varepsilon\|_{L^2(\Omega)}
\le C\sqrt{\varepsilon}\|u_0\|_{H^2(\Omega)},
\end{equation}
where the last inequality follows from (\ref{estimate-w}). By  combining (\ref{e3.9}), (\ref{e3.10}) and (\ref{e3.11}), we have proved
that
$$
\| p_\varepsilon-p_0-\Big[\pi^\varepsilon S_\varepsilon \left(\nabla\widetilde{u}_0\right)-\average_\Omega \pi^\varepsilon S_\varepsilon \left(\nabla\widetilde{u}_0\right)\Big]\|_{L^2(\Omega)}\le C\sqrt{\varepsilon}\|u_0\|_{H^2(\Omega)}.
$$
This completes the proof.
\end{proof}


\section{Convergence rates for $u_\varepsilon$ in $L^2$}
\setcounter{equation}{0}

To establish the sharp  $O(\varepsilon)$ rate for $u_\varepsilon$ in $L^2$, in view of (\ref{elemma3.2}),  we obtain 
$$
\|u_\varepsilon-u_0-\varepsilon \chi^\varepsilon S_\varepsilon \nabla \widetilde{u}_0+w_\varepsilon\|_{L^2(\Omega)}
 \le C\varepsilon\|u_0\|_{H^2(\Omega)}.
$$
Using Proposition \ref{prop2.2} and (\ref{extension}),
$$
\|\chi^\varepsilon S_\varepsilon\nabla \widetilde{u}_0\|_{L^2(\Omega)}\le C\|\chi\|_{L^2(Y)}\|\nabla \widetilde{u}_0\|_{L^2(\mathbb{R}^{d})}\le C\| u_0\|_{H^2(\Omega)}.
$$
Thus, 
\begin{equation}\label{e4.1}
\|u_\varepsilon-u_0\|_{L^2(\Omega)} \le C\varepsilon\|u_0\|_{H^2(\Omega)}+\|w_\varepsilon\|_{L^2(\Omega)},
\end{equation}
and it remains to estimate $\|w_\varepsilon\|_{L^2(\Omega)}$.

\begin{lemma}\label{lemma4.1}
Let $\Omega$ be a bounded $C^{1,1}$ domain. Suppose that $A$ satisfies ellipticity condition (\ref{Ellipticity})
and periodicity condition (\ref{Periodicity}). Given 
$g\in H^1(\Omega)$ and $f\in H^{3/2}(\partial\Omega;\mathbb{R}^d)$ satisfying the compatibility condition (\ref{Compatibility}), 
for $F\in L^2(\Omega;\mathbb{R}^d)$, let 
$(u_\varepsilon,p_\varepsilon)$, $(u_0,p_0)$ be weak solutions of the Dirichlet problems (\ref{DirichletStokes}), (\ref{DirichletStokes0}), respectively.
Then
\begin{equation}\label{e4.2}
\|u_\varepsilon-u_0-\varepsilon (1-\widetilde{\theta}_\varepsilon)\chi^\varepsilon S_\varepsilon \left(\nabla\widetilde{u}_0\right)\|_{H^1(\Omega)} \le C\sqrt{\varepsilon}\|u_0\|_{H^2(\Omega)},
\end{equation}
and
\begin{equation}\label{e4.3}
\|p_\varepsilon-p_0-\Big[(1-\widetilde{\theta}_\varepsilon)\pi^\varepsilon S_\varepsilon \left(\nabla\widetilde{u}_0\right)
-\average_\Omega \pi^\varepsilon 
S_\varepsilon \left(\nabla\widetilde{u}_0\right)\Big]\|_{L^2(\Omega)}\le C\sqrt{\varepsilon}\|u_0\|_{H^2(\Omega)},
\end{equation}
where $C$ depends only on $d$, $\mu$, and $\Omega$.
\end{lemma}

\begin{proof}
Note that
\begin{equation}\label{plemma4.1}
\aligned
\|\varepsilon \widetilde{\theta}_\varepsilon \chi^\varepsilon S_\varepsilon \nabla \widetilde{u}_0\|_{H^1(\Omega)} 
&\le C \varepsilon \| \chi^\varepsilon S_\varepsilon(\nabla \widetilde{u}_0)\|_{L^2(\Omega_{2\varepsilon})}
+C \varepsilon \| \chi^\varepsilon S_\varepsilon(\nabla^2 \widetilde{u}_0)\|_{L^2(\Omega_{2\varepsilon})}\\
&\qquad
+ C \| (|\chi^\varepsilon|+|(\nabla \chi)^\varepsilon |) S_\varepsilon (\nabla \widetilde{u}_0)\|_{L^2(\Omega_{2\varepsilon})}\\
&\le C\sqrt{\varepsilon}\|u_0\|_{H^2(\Omega)},
\endaligned
\end{equation}
where we have used Lemma \ref{lemma4.0} and Proposition (\ref{prop2.2}) for the last inequality.
This, together with estimate (\ref{ethm1.1a}), gives (\ref{e4.2}).

Similarly, using Lemma \ref{lemma4.0}, we see that
$$
\begin{aligned}
&\|\widetilde{\theta}_\varepsilon \pi^\varepsilon S_\varepsilon \nabla \widetilde{u}_0\|^2_{L^2(\Omega)} 
\le C\int_{(\partial\Omega)_{2\varepsilon}} |\pi^\varepsilon S_\varepsilon(\nabla \widetilde{u}_0)|^2
 \le C\varepsilon\|u_0\|^2_{H^2(\Omega)}.
\end{aligned}
$$
This, together with estimate (\ref{ethm1.1b}), gives (\ref{e4.3}).
\end{proof}

\begin{proof}[\textbf{Proof of Theorem \ref{theorem1.2}}]
In view of (\ref{e4.1}), it suffices to show that $$\| w_\varepsilon\|_{L^2(\Omega)} \le C \varepsilon \| u_0\|_{H^2(\Omega)}.$$
Furthermore, let
$$
\phi_\varepsilon= \varepsilon \theta_\varepsilon \chi^\varepsilon S_\varepsilon \nabla \widetilde{u}_0.
$$
Since $\|\phi_\varepsilon\|_{L^2(\Omega)}\le C\varepsilon\|u_0\|_{H^2(\Omega)}$, it is enough to show that
\begin{equation}\label{estimate-eta}
\|\eta_\varepsilon \|_{L^2(\Omega)} \le C \varepsilon \| u_0\|_{H^2(\Omega)},
\end{equation}
where
$
\eta_\varepsilon= w_\varepsilon -\phi_\varepsilon.
$

To this end, we first note that by the definition of $(w_\varepsilon,\tau_\varepsilon)$ in (\ref{w}),  
the functions $(\eta_\varepsilon, \tau_\varepsilon) \in H^1_0(\Omega;\mathbb{R}^d) \times L^2(\Omega)$ satisfy
\begin{equation}\label{e4.4}
\left\{
\begin{aligned}
\mathcal{L}_\varepsilon (\eta_\varepsilon) +\nabla \tau_\varepsilon&
=-\mathcal{L}_\varepsilon \phi_\varepsilon & \quad  & \text{ in } \Omega,\\
\text{div }\eta_\varepsilon&= \varepsilon\, \text{div }((1-\theta_\varepsilon)\chi^\varepsilon S_\varepsilon \nabla \widetilde{u}_0)&\quad & \text{ in } \Omega,\\
\eta_\varepsilon&=0 &\quad & \text{ on }\partial \Omega.
\end{aligned}
\right.
\end{equation}
Let $(\eta_0,\tau_0) \in H^1_0(\Omega;\mathbb{R}^d) \times L^2(\Omega)$  be a weak solution of the homogenized Dirichlet problem
\begin{equation}\label{e4.5}
\left\{
\begin{aligned}
\mathcal{L}_0 (\eta_0)+\nabla \tau_0&=-\mathcal{L}_\varepsilon \phi_\varepsilon  &\quad & \text{ in }\Omega,\\
\text{div }\eta_0&= \varepsilon\, \text{div }((1-\theta_\varepsilon)\chi^\varepsilon S_\varepsilon \nabla \widetilde{u}_0)&\quad
 & \text{ in }\Omega,\\
\eta_0&=0 &\quad & \text{ on }\partial \Omega.\\
\end{aligned}
\right.
\end{equation}
To estimate $\eta_\varepsilon -\eta_0$,
we consider the following duality problems. 
 For any $H\in L^2(\Omega;\mathbb{R}^d)$, 
let $(\rho_\varepsilon,\sigma_\varepsilon)\in H^1_0(\Omega;\mathbb{R}^d) \times L^2(\Omega)$ be the weak solution of
\begin{equation}\label{e4.6}
\left\{
\begin{aligned}
\mathcal{L}^*_\varepsilon (\rho_\varepsilon) + \nabla \sigma_\varepsilon&=H&\quad  & \text{ in }\Omega,\\
\text{div }\rho_\varepsilon&= 0 & \quad & \text{ in }\Omega,\\
\rho_\varepsilon&=0 & \quad & \text{ on }\partial \Omega,\\
\end{aligned}
\right.
\end{equation}
and $(\rho_0,\sigma_0)\in (H^2(\Omega;\mathbb{R}^d)\cap H^1_0(\Omega;\mathbb{R}^d)) \times H^1(\Omega)$  the weak solution of
\begin{equation}\label{e4.7}
\left\{
\begin{aligned}
\mathcal{L}^*_0 (\rho_0)+\nabla \sigma_0&=H  & \quad & \text{ in }\Omega,\\
\text{div }\rho_0&= 0 & \quad &\text{ in }\Omega,\\
\rho_0&=0 & \quad & \text{ on }\partial \Omega,\\
\end{aligned}
\right.
\end{equation}
with
$$
\int_\Omega \sigma_\varepsilon=\int_\Omega \sigma_0=0.
$$
Here we have used the notation: $\mathcal{L}^*_\varepsilon =-\text{div} \big(A^*(x/\varepsilon)\nabla\big)$
and $\mathcal{L}_0^* =-\text{div}\big(\widehat{A^*}\nabla \big)$.
We note that  Lemma \ref{lemma4.1} continues to hold for $\mathcal{L}^{*}_\varepsilon$, as 
$A^*$ satisfies the same conditions as $A$. 
Also, by the $W^{2,2}$ estimates (\ref{reg}) for Stokes systems with constant coefficients in $C^{1,1}$ domains,
$$
 \|\rho_0\|_{H^2(\Omega)} +\|\sigma_0\|_{H^1(\Omega)} \le C\, \| H\|_{L^2(\Omega)}.
$$
As a result, we have 
\begin{equation}\label{e4.8}
\|\rho_\varepsilon-\rho_0-\varepsilon (1-\widetilde{\theta}_\varepsilon)\chi^{*\varepsilon} 
S_\varepsilon\left(\nabla \widetilde{\rho}_0\right)\|_{H^1(\Omega)} 
 \le C\sqrt{\varepsilon}\|\rho_0\|_{H^2(\Omega)}\\
\le C\sqrt{\varepsilon}\|H\|_{L^2(\Omega)},
\end{equation}
and
\begin{equation}\label{e4.9}
\|\sigma_\varepsilon-\sigma_0-\Big[(1-\widetilde{\theta}_\varepsilon)\pi^{*\varepsilon} 
S_\varepsilon \left(\nabla\widetilde{\rho}_0\right)-\average_\Omega \pi^{*\varepsilon} 
S_\varepsilon \left(\nabla\widetilde{\rho}_0\right)\Big]\|_{L^2(\Omega)}\le C\sqrt{\varepsilon}\|H\|_{L^2(\Omega)},
\end{equation}
where $(\chi^*,\pi^*)$ denotes  the correctors associated with the adjoint matrix  $A^*$.

Let $\Psi=-\mathcal{L}_\varepsilon \phi_\varepsilon$, and 
 $$
 \Gamma=\text{div }(\varepsilon(1-\theta_\varepsilon)\chi^\varepsilon S_\varepsilon \nabla \widetilde{u}_0).
 $$
 Note that by (\ref{e4.4}), (\ref{e4.5}), (\ref{e4.6}) and (\ref{e4.7}),
\begin{equation}\label{e4.10}
\aligned
\int_\Omega H\cdot (\eta_\varepsilon-\eta_0)  &=\langle \Psi, \rho_\varepsilon-\rho_0\rangle_{H^{-1}(\Omega;\mathbb{R}^d)\times H_0^1(\Omega;\mathbb{R}^d)}-\int_\Omega \Gamma(\sigma_\varepsilon-\sigma_0)\\
&=J_1+J_2.
\endaligned
\end{equation}
For the first term of the RHS of (\ref{e4.10}), because $\Psi\in H^{-1}(\Omega;\mathbb{R}^d)$ is  supported in $(\partial\Omega)_{\varepsilon}$, and $1-\widetilde{\theta}_\varepsilon=0$ in $(\partial \Omega)_\varepsilon$,
we obtain 
\begin{equation*}
J_1=\langle \Psi ,\rho_\varepsilon-\rho_0-\varepsilon(1-\widetilde{\theta}_\varepsilon)\chi^{*\varepsilon} S_\varepsilon\left(\nabla\widetilde{\rho}_0\right)\rangle_{H^{-1}(\Omega;\mathbb{R}^d)\times H_0^1(\Omega;\mathbb{R}^d)}.
\end{equation*}
Therefore,
\begin{equation}\label{e4.11}
\begin{aligned}
|J_1|&\le \|\Psi\|_{H^{-1}(\Omega)}\|\rho_\varepsilon-\rho_0-\varepsilon(1-\widetilde{\theta}_\varepsilon)\chi^{*\varepsilon} S_\varepsilon\left(\nabla\widetilde{\rho}_0\right)\|_{H^1(\Omega)}\\
&\le C\|\varepsilon \theta_\varepsilon\chi^\varepsilon 
S_\varepsilon\nabla \widetilde{u}_0\|_{H^1(\Omega)}\, \sqrt{\varepsilon}\|H\|_{L^2(\Omega)}\\
&\le C {\varepsilon}\|u_0\|_{H^2(\Omega)}\, \|H\|_{L^2(\Omega)}
\end{aligned}
\end{equation}
where the second inequality follows from (\ref{e4.8}), and the last inequality follows from the analog of (\ref{plemma4.1}) (with $\widetilde{\theta}_\varepsilon$ replaced by $\theta_\varepsilon$). For the second term of the RHS of (\ref{e4.10}), we recall that $\text{div }(\chi)=0$. Hence,
\begin{equation*}
\Gamma=-\varepsilon\frac{\partial \theta_\varepsilon}{\partial x_\alpha}
\chi_j^{\alpha\beta}(x/\varepsilon)S_\varepsilon\frac{\partial \widetilde{u}_0^\beta}{\partial x_j}+\varepsilon(1-\theta_\varepsilon)\chi_j^{\alpha\beta}(x/\varepsilon)S_\varepsilon\frac{\partial^2 \widetilde{u}_0^\beta}{\partial x_\alpha \partial x_j}=\Gamma_1+\Gamma_2.
\end{equation*}
Since $\int_\Omega \Gamma =0$,  for any constant $E$,
$$
J_2=-\int_\Omega \Gamma(\sigma_\varepsilon-\sigma_0+E)=-\int_\Omega [\Gamma_1+\Gamma_2](\sigma_\varepsilon-\sigma_0+E).
$$
We split $J_2$ as two integrals, for the first integral, again since $1-\widetilde{\theta}_\varepsilon=0$ in $(\partial\Omega)_\varepsilon$ and $\Gamma_1$ is supported in $(\partial\Omega)_\varepsilon$, just as we did for $J_1$, 
$$
-\int_\Omega \Gamma_1(\sigma_\varepsilon-\sigma_0+E)=-\int_{\Omega} \Gamma_1\Big(\sigma_\varepsilon-\sigma_0-(1-\widetilde{\theta}_\varepsilon)\pi^{*\varepsilon}S_\varepsilon \left(\nabla\tilde{\rho}_0\right)+E\Big).
$$
Now, if we choose the constant $E$ as $E=\average_\Omega \pi^{*\varepsilon}S_\varepsilon \left(\nabla\tilde{\rho}_0\right)$, then
\begin{equation}\label{e4.12}
\begin{aligned}
&\bigg|\int_\Omega \Gamma_1 (\sigma_\varepsilon-\sigma_0+E)\bigg|\\
&=\bigg|\int_{\Omega} \Gamma_1 \bigg\{\sigma_\varepsilon-\sigma_0-\Big[(1-\widetilde{\theta}_\varepsilon)\pi^{*\varepsilon} S_\varepsilon \left(\nabla\widetilde{\rho}_0\right)-\average_\Omega 
\pi^{*\varepsilon} S_\varepsilon \left(\nabla\widetilde{\rho}_0\right)\Big]\bigg\}\bigg|\\
&\le C\|\Gamma_1\|_{L^2((\partial\Omega)_\varepsilon)}\, \sqrt{\varepsilon}\|H\|_{L^2(\Omega)}\\
&\le C\big(\sqrt{\varepsilon}\|\chi\|_{L^2(Y)}\|\nabla \widetilde{u}_0\|_{H^1(\mathbb{R}^d)}\big)
\big(\sqrt{\varepsilon}\|H\|_{L^2(\Omega)}\big)\\
&\le C\varepsilon\|u_0\|_{H^2(\Omega)}\|H\|_{L^2(\Omega)},
\end{aligned}
\end{equation}
where we have used (\ref{e4.9}) and Lemma \ref{lemma4.0}.
For the second integral in $J_2$, we have
\begin{equation}\label{e4.13}
\begin{aligned}
& \bigg|\int_\Omega \Gamma_2 \big(\sigma_\varepsilon-\sigma_0
+E\big)\bigg| \\
&\le \|\Gamma_2\|_{L^2(\Omega)}\|\sigma_\varepsilon-\sigma_0+\average_\Omega \pi^{*\varepsilon}S_\varepsilon \left(\nabla\widetilde{\rho}_0\right)\|_{L^2(\Omega)}\\
&\le C\varepsilon\|u_0\|_{H^2(\Omega)}\|H\|_{L^2(\Omega)},
\end{aligned}
\end{equation}
where for the last inequality we have used  
$$
\begin{aligned}
\|\sigma_\varepsilon-\sigma_0+\average_\Omega \pi^{*\varepsilon}S_\varepsilon \left(\nabla\widetilde{\rho}_0\right)\|_{L^2(\Omega)}
&\le \|\sigma_\varepsilon\|_{L^2(\Omega)}+\|\sigma_0\|_{L^2(\Omega)}
+\|\average_\Omega \pi^{*\varepsilon} S_\varepsilon \left(\nabla\widetilde{\rho}_0\right)\|_{L^2(\Omega)}\\
&\le C\|H\|_{L^2(\Omega)}.
\end{aligned}
$$
Therefore, by combining (\ref{e4.11})-(\ref{e4.13}), we have proved
\begin{equation}\label{e4.14}
\left|\int_\Omega H(\eta_\varepsilon-\eta_0)\right|\le C\varepsilon\|u_0\|_{H^2(\Omega)}\|H\|_{L^2(\Omega)} \quad\text{ for any }H\in L^2(\Omega;\mathbb{R}^d).
\end{equation}
By duality this implies that
\begin{equation}\label{e4.15}
\|\eta_\varepsilon-\eta_0\|_{L^2(\Omega)}\le C\varepsilon\|u_0\|_{H^2(\Omega)}.
\end{equation}

Finally, the problem has been reduced to the estimate of $\|\eta_0\|_{L^2(\Omega)}$.
 This will be done by another duality argument. Let $(\rho_0, \sigma_0)$
 be defined by (\ref{e4.7}). Then
\begin{equation}\label{e4.16}
\begin{aligned}
&\left|\int_\Omega H\cdot \eta_0\right|
=\left| \langle \Psi, \rho_0\rangle_{H^{-1}(\Omega; \mathbb{R}^d)\times H^1_0(\Omega; \mathbb{R}^d)}
-\int_\Omega \Gamma\sigma_0\right|\\
&\le  | \langle \Psi, \rho_0\rangle_{H^{-1}(\Omega; \mathbb{R}^d)\times H^1_0(\Omega; \mathbb{R}^d)}|
+\left|\int_ {\Omega_\varepsilon} \Gamma_1\sigma_0\right|+\left|\int_\Omega \Gamma_2\sigma_0\right|\\
&=K_1+K_2+K_3,
\end{aligned}
\end{equation}
where $\Psi, \Gamma, \Gamma_1$ and $\Gamma_2$ are as denoted above. 
Notice that again by Lemma \ref{lemma4.0} and the analog of (\ref{plemma4.1}) (with $\widetilde{\theta}_\varepsilon$ replaced by $\theta_\varepsilon$), we have
\begin{equation}\label{e4.17}
\begin{aligned}
K_1&\le \|\Psi\|_{H^{-1}(\Omega)}\|\rho_0\|_{H^1(\Omega_\varepsilon)}\\
&\le C\|\varepsilon\theta_\varepsilon \chi^\varepsilon S_\varepsilon \nabla\widetilde{u}_0\|_{H^1(\Omega)}\sqrt{\varepsilon}\| \rho_0\|_{H^2(\Omega)}\\
& \le C(\sqrt{\varepsilon}\| u_0\|_{H^2(\Omega)})(\sqrt{\varepsilon}\| \rho_0\|_{H^2(\Omega)}) \\
&\le C\varepsilon\|u_0\|_{H^2(\Omega)}\|H\|_{L^2(\Omega)}.
\end{aligned}
\end{equation}
Similarly, again by Lemma \ref{lemma4.0},
\begin{equation}\label{e4.18}
\begin{aligned}
K_2 &\le \|\Gamma_1\|_{L^2((\partial\Omega)_\varepsilon)}\|\sigma_0\|_{L^2(\Omega_\varepsilon)}\\
&\le C(\sqrt{\varepsilon}\|\chi\|_{L^2(Y)}\| \widetilde{u}_0\|_{H^2(\mathbb{R}^d)})
(\sqrt{\varepsilon}\|\sigma_0\|_{H^1(\Omega))})\\
&\le C\varepsilon\|u_0\|_{H^2(\Omega)}\|H\|_{L^2(\Omega)},
\end{aligned}
\end{equation}
and
\begin{equation}\label{e4.19}
K_3\le \|\Gamma_2\|_{L^2(\Omega)}\|\sigma_0\|_{L^2(\Omega)}
\le C\varepsilon\|u_0\|_{H^2(\Omega)}\|H\|_{L^2(\Omega)}.
\end{equation}
By combining (\ref{e4.17})-(\ref{e4.19}), we obtain 
$$
\left|\int_\Omega H\cdot \eta_0\right|\le C\varepsilon\|u_0\|_{H^2(\Omega)}\|H\|_{L^2(\Omega)},
$$
which, by duality, leads to
\begin{equation}\label{e4.20}
\|\eta_0\|_{L^2(\Omega)} \le C\varepsilon\|u_0\|_{H^2(\Omega)}.
\end{equation}
Hence we have proved that
\begin{equation}
\|w_\varepsilon\|_{L^2(\Omega)}\le \|\eta_\varepsilon-\eta_0\|_{L^2(\Omega)}+\|\eta_0\|_{L^2(\Omega)}+\|\phi_\varepsilon\|_{L^2(\Omega)}
\le C\varepsilon\|u_0\|_{H^2(\Omega)}.
\end{equation}
The proof is finished.
\end{proof}

\bibliographystyle{plain}

\bibliography{lib}

\medskip

\begin{flushleft}
Shu Gu,
Department of Mathematics,
University of Kentucky,
Lexington, Kentucky 40506,
USA.

E-mail: gushu0329@uky.edu
\end{flushleft}

\medskip

\noindent \today

\end{document}